\documentclass{amsart}%
\usepackage{amsmath}%
\usepackage{amsfonts}%
\usepackage{amssymb}%
\usepackage{graphicx}

\newtheorem{theorem}{Theorem}
\theoremstyle{plain}

\newtheorem{corollary}{Corollary}

\newtheorem{definition}{Definition}
\newtheorem{example}{Example}

\newtheorem{proposition}{Proposition}

\numberwithin{equation}{section}

\begin{document}
\title{Existence of Life in Lenia}
\author{Craig Calcaterra \& Axel Boldt}
\address{Department of Mathematics and Statistics, Metropolitan State University, Saint Paul, Minnesota, USA}
\email{craig.calcaterra @ metrostate dot edu}
\date{March 27, 2022}

\begin{abstract}
Lenia is a continuous generalization of Conway's Game of Life. Bert
Wang-Chak Chan has discovered and published many seemingly organic dynamics in
his Lenia simulations since 2019. These simulations follow the Euler curve
algorithm starting from function space initial conditions. The
Picard-Lindel\"{o}f Theorem for the existence of integral curves to Lipschitz
vector fields on Banach spaces fails to guarantee solutions, because the
vector field associated with the integro-differential equation defining Lenia
is discontinuous. However, we demonstrate the dynamic Chan is using to
generate simulations is actually an arc field and not the traditional Euler
method for the vector field derived from the integro-differential equation.
Using arc field theory we prove the Euler curves converge to a unique flow
which solves the original integro-differential equation. Extensions are
explored and the modeling of entropy is discussed.

\end{abstract}

\maketitle

\section{Introduction}

Lenia is a continuous model proposed by Bert Wang-Chak Chan around 2018
\cite{Chan1} which generalizes the Game of Life. The Game of Life (GoL) is a
discrete dynamical system, proposed by Conway in 1970. Lenia follows simpler
generalizations of GoL, such as ``Larger than Life'' \cite{Evans},
\cite{Evans2}, ``RealLife'' \cite{Pivato}, and ``SmoothLife'' \cite{Rafler}.

The original Game of Life is a cellular automata, a discrete dynamical system
on $\mathbb{Z}^{2}$. Its rules can be written with convolution on the state
space $\mathbb{Z}^{2}$ in the following way. In this form the rules are then
straightforward to generalize to continuous dynamical systems. Denote the
Chebyshev distance, or $L^{\infty}$ metric, on $\mathbb{Z}^{2}$ by $\left\|
x\right\|  _{\infty}:=\underset{i\in\left\{  1,2\right\}  }{\max}\left|
x_{i}\right|  $. Define the kernel $\mathcal{K}:\mathbb{Z}^{2}\rightarrow
\mathbb{R}$ by%
\[
\mathcal{K}\left(  x\right)  :=\left\{
\begin{array}
[c]{cc}%
1 & \text{if }\left\|  x\right\|  _{\infty}\leq1\\
0 & \text{otherwise.}%
\end{array}
\right.
\]
For another function $g:\mathbb{Z}^{2}\rightarrow\mathbb{R}$ define the
discrete convolution operator $\ast$ by $\left(  \mathcal{K}\ast g\right)
:\mathbb{Z}^{2}\rightarrow\mathbb{R}$ with%
\[
\left(  \mathcal{K}\ast g\right)  \left(  x\right)  :=\underset{y\in
\mathbb{Z}^{2}}{\sum}\mathcal{K}\left(  y\right)  g\left(  x-y\right)
\]
Then the discrete dynamical system $f:\mathbb{Z}^{2}\times\mathbb{N\rightarrow
}\left\{  0,1\right\}  $ for GoL is given by%
\begin{equation}
f_{t+1}\left(  x\right)  =\left\{
\begin{array}
[c]{ll}%
1 & \text{if }f_{t}\left(  x\right)  =0\text{ and }\left(  \mathcal{K}\ast
f_{t}\right)  \left(  x\right)  =3\text{ (birth)}\\
1 & \text{if }f_{t}\left(  x\right)  =1\text{ and }\left(  \mathcal{K}\ast
f_{t}\right)  \left(  x\right)  \in\left\{  3,4\right\}  \text{ (survival)}\\
0 & \text{otherwise (death).}%
\end{array}
\right. \label{LineGoLInstructions}%
\end{equation}
for $t\in\mathbb{N}$ and $x\in\mathbb{Z}^{2}$. The initial condition is given
by an arbitrary function $f_{0}:\mathbb{Z}^{2}\rightarrow\left\{  0,1\right\}
$.

What has made GoL so famous is that these finite, simple rules are enough to
create infinitely complex, sometimes even beautiful, dynamics. In his public
introduction of the GoL instructions in 1970, Martin Gardner \cite{Gardner}
wrote, ``Because of Life's analogies with the rise, fall and alterations of a
society of living organisms, it belongs to a growing class of what are called
`simulation games' (games that resemble real-life processes).'' Thanks to the
existence of ``gliders'' in GoL, Conway showed it is possible to carefully set
the initial conditions so that iterating $f$ generates structures with the
behaviors of the three logical gates, AND, OR, NOT. This proves GoL is a
universal Turing machine, and therefore capable of approximating any
mathematical operation. Therefore the most complex mathematical simulations
imaginable may be faithfully represented by GoL with the right choice of
initial condition.

GoL is a recursive function on $\mathbb{Z}^{2},$ and simply being well-defined
is enough to guarantee the existence of the sequence of steps starting with
any initial condition. So it is obvious ``life exists'' in GoL starting from
any initial condition. But it is not so obvious solutions exist for its
continuous generalizations.

Chan's instructions generalize Conway's by using decreasingly-small, discrete
state space locations and decreasingly-small time steps until, in the limit,
the instructions hopefully converge to a continuous dynamical system. The
resulting simulation is intended to mimic some of the aspects of actual
material life. Like GoL, Lenia appears to give complex dynamics with
relatively simple instructions.

However, the existence of the limiting processes for these continuous
instructions are difficult to prove. In fact, we show in this paper that the
instructions for Lenia are an integro-differential equation which gives a
discontinuous vector field on an infinite dimensional metric space. The basic
existence theorem for solutions to ODEs on Banach spaces, the
Picard-Lindel\"{o}f Theorem, requires Lipschitz continuous vector fields, and
is therefore not applicable.

Instead of using a vector field on a Banach space, we generalize to an arc
field on a metric space \cite{CalcBleecker} and show solutions exist there.
Then we prove that the solution to the arc field solves the original
integro-differential equation, thus proving the existence of the generalized
Game of Life dynamics in Lenia.

Two major advantages arise from this effort to extend the discrete model to
the continuous. First, there is the possibility of greater descriptive power
and more complicated dynamics in a continuous model. Second, the continuous
limit will have computable laws available, such as hard limits on measurements
of the total mass or higher order moments under particular parameter choices.
There are more theoretical results for continuous dynamical systems than
discrete systems, because of the control that the continuum $\mathbb{R}$ gives
compared with the gaps in discrete $\mathbb{Z}$.

\bigskip

In the next section we will review Chan's Lenia model, and compare it to GoL.
In the third section we interpret Chan's instructions as an
integro-differential equation. This gives rise to a vector field which,
however, is discontinuous and incompatible with the Euler curve approach Chan
uses to simulate the Lenia dynamic. We instead define an arc field and show
that Chan's simulation is the Euler approximation of solutions to this arc
field. The main result of the paper is to show that the Euler approximations
will converge to a solution of the integro-differential equation.

To achieve this we first review the theory of arc fields in Section 5, using
it to prove the existence and uniqueness of the limit of the Lenia algorithm
in Section 7. Then we show the resulting solutions also solve Lenia's
integro-differential equation in Section 8. In the final section, we build on
those solutions to create more complicated models. We feed the creatures food,
and put them in competition with each other using generalizations of the
Lotka--Volterra predator-prey ODEs.

\section{Lenia: continuous version of GoL}

Chan's formulation \cite[(17), (18), p. 266]{Chan1} of the Lenia dynamic
$f:\mathbb{R}^{n}\times\mathbb{R}\rightarrow\left[  0,1\right]  $ where
$\left(  x,t\right)  \longmapsto f_{t}\left(  x\right)  $ for $x\in
\mathbb{R}^{n} $ and $t\in\left[  0,\infty\right)  $ is given by%
\begin{equation}
f_{t+dt}\left(  x\right)  =f_{t}\left(  x\right)  +dt\left[  G(K\ast
f_{t})\left(  x\right)  \right]  _{-f_{t}\left(  x\right)  /dt}^{\left(
1-f_{t}\left(  x\right)  \right)  /dt}\label{LineChanFormulation}%
\end{equation}
for positive infinitesimal $dt$. We discuss this formula in the next two paragraphs.

Here $K$ is a kernel function $K:\mathbb{R}^{n}\rightarrow\mathbb{R}$ which is
assumed to be $L^{1}$ so the continuous convolution%
\[
\left(  K\ast f_{t}\right)  \left(  x\right)  :=\int_{\mathbb{R}^{n}}K\left(
x-y\right)  f_{t}\left(  y\right)  dy
\]
is well defined. $G:\mathbb{R}\rightarrow\mathbb{R}$ is the activation
function which we will assume to be bounded and Lipschitz continuous with
constant $C_{G}\geq0$ such that $\left|  G\left(  x\right)  -G\left(
y\right)  \right|  \leq C_{G}\left|  x-y\right|  $. Finally, the square
brackets denote the clip function%
\[
\left[  x\right]  _{a}^{b}:=\left\{
\begin{array}
[c]{cc}%
b & \text{if }x\geq b\\
x & \text{if }a<x<b\\
a & \text{if }x\leq a
\end{array}
\right.  \text{.}%
\]

Basic choices of kernel $K$ and activation function $G$ used to create
complicated dynamics satisfy $K:\mathbb{R}^{2}\rightarrow\mathbb{R}^{+}$ such
as $K\left(  x,y\right)  :=\exp\left(  4-\frac{1}{r\left(  1-r\right)
}\right)  1_{\left\{  0<r<1\right\}  }\left(  x,y\right)  $ where
$r=\sqrt{x^{2}+y^{2}}$and $1_{S}$ is the indicator function on the set $S$,
and $G:\left[  0,1\right]  \rightarrow\left[  -1,1\right]  $ such as $G\left(
x\right)  :=2\exp\left(  -\left(  x-\frac{1}{2}\right)  ^{2}/2\right)  -1$.
Such choices were initially explored by Chan in \cite{Chan1}.

Line $\left(  \ref{LineChanFormulation}\right)  $ is meant to describe the
time evolution of the creatures of Lenia. The creature at time $t$ is
described by the function $f_{t}:\mathbb{R}^{n}\rightarrow\left[  0,1\right]
$ where usually $n=2$ or $3$.

For a given a creature $f$ the value $f_{t}\left(  x\right)  \in\left[
0,1\right]  $ represents the density of mass of the creature at the location
in space $x\in\mathbb{R}^{n}$ at time $t$.

How do we understand the meaning of the kernel and activation functions? The
metaphor is not completely robust, because we are using an extremely simple
mathematical model to attempt the simulation of the complexity of organic
life, but the intuitive idea may be something like the following. The initial
condition of a creature $f_{0}$ represents its physical state at birth. The
kernel $K$ dictates how a creature with initial condition responds to its
environment. And the activation function $G$ encodes how the environment
supports creatures. So $K$ synthesizes and simplifies a great deal of
information if we think of these creatures as simulations of biological
organisms. This includes simplifying the result of the creature's DNA acting
toward expressing its morphology in response to its sustenance and neighbors
and other environmental factors. In sum, if we change the formula for $K$ then
we change the species of creature we are studying. If we change the initial
condition $f_{0}$ we change the individual within that species. Changes in $G$
may be seen as reflecting the ability of the environment to support certain
shapes of individuals. Recently in \cite{Hamon} complicated dynamics have been
discovered for kernels of the form%
\[
K\left(  x,y\right)  :=\overset{k}{\underset{i=1}{\sum}}b_{i}\exp\left(
-\frac{\left(  \frac{r}{c}-a_{i}\right)  ^{2}}{2w_{i}}\right)  \text{.}%
\]
There they used neural networks to explore parameter space $\left\{
a_{i},b_{i},w_{i}\right\}  $ for fixed $k$ and $c$ to discover initial
conditions which simulate adaptivity to obstacles in their environment.

\bigskip

Chan's formulation can be thought of as a continuous generalization of the
original GoL instructions in the following way (\textit{cf.} \cite[p.
258]{Chan1}). The original GoL dynamical system $f:\mathbb{Z}^{2}%
\times\mathbb{N\rightarrow}\left\{  0,1\right\}  $ can be rewritten as%
\[
f_{t+1}\left(  x\right)  =\left[  f_{t}\left(  x\right)  +G\left(  K\ast
f_{t}\right)  \left(  x\right)  \right]  _{0}^{1}%
\]
where the square brackets denote the clip function $\left[  x\right]  _{0}%
^{1}:=\min\left\{  \max\left\{  x,0\right\}  ,1\right\}  $ and the kernel
$K:\mathbb{Z}^{2}\rightarrow\left\{  0,1\right\}  $ is given by%
\[
K\left(  x\right)  :=\left\{
\begin{array}
[c]{cc}%
1 & \text{if }\left\|  x\right\|  _{\infty}=1\\
1/2 & \text{if }\left\|  x\right\|  _{\infty}=0\\
0 & \text{if }\left\|  x\right\|  _{\infty}>1
\end{array}
\right.
\]
and the activation function $G:\mathbb{R\rightarrow R}$ is given by%
\[
G\left(  x\right)  =2\cdot1_{\left[  2.5,3.5\right]  }\left(  x\right)
-1\text{.}%
\]

\section{Integro-differential equation}

Formula $\left(  \ref{LineChanFormulation}\right)  $ is not a traditional
differential because of the infinitesimals in the bounds of the clip function.
By replacing the infinitesimals $dt$ with $\Delta t>0$ and taking the limit,
we interpret Chan's formula as a forward integro-differential equation as
follows:%
\[
\frac{d}{dt^{+}}f_{t}\left(  x\right)  =\underset{\Delta t\rightarrow0^{+}%
}{\lim}\frac{f_{t+\Delta t}\left(  x\right)  -f_{t}\left(  x\right)  }{\Delta
t}=\underset{\Delta t\rightarrow0^{+}}{\lim}\left[  G(\left(  K\ast
f_{t}\right)  (x))\right]  _{-f_{t}(x)/\Delta t}^{\left(  1-f_{t}(x)\right)
/\Delta t}%
\]
so the Lenia formula is%
\begin{equation}
\frac{d}{dt^{+}}f_{t}\left(  x\right)  =V\left(  f_{t}\right)  (x):=\left\{
\begin{array}
[c]{cc}%
G(K\ast f_{t}(x)) & \text{if }0<f_{t}(x)<1\\
\left[  G(K\ast f_{t}(x))\right]  _{0}^{\infty} & \text{if }f_{t}\left(
x\right)  =0\\
\left[  G(K\ast f_{t}(x))\right]  _{-\infty}^{0} & \text{if }f_{t}\left(
x\right)  =1\text{.}%
\end{array}
\right. \label{LineIntegroDiffEQ1}%
\end{equation}
Line \ref{LineIntegroDiffEQ1} is an integro-differential equation because of
the convolution $K\ast f_{t}$.

It is tempting to interpret formula \ref{LineIntegroDiffEQ1} as a vector field
$V$ on a Banach space then attempt to use traditional ODE theory to show the
existence of integral curves to $V$. There are several problems with this
effort. The natural Banach space to use would be $L^{\infty}\left(
\mathbb{R}^{n},\mathbb{R}\right)  $. Then $V$ is defined on the closed
subspace $L^{\infty}\left(  \mathbb{R}^{n},\left[  0,1\right]  \right)  $
which might be extended easily enough. However, the derivative in
\ref{LineIntegroDiffEQ1} is taken pointwise and not in the Banach space sense.
In fact you can't make a vector field on $L^{\infty}$ specifying $V\left(
f\right)  $ by referring to particular values $f\left(  x\right)  $ at
particular positions $x\in\mathbb{R}^{n}$ because the members of $L^{\infty}$
are not merely functions but equivalence classes of functions up to measure
$0$. To be careful, we immediately need to pick a less commonly used Banach space.

We take as our context the Banach space $\mathcal{B}\left(  \mathbb{R}%
^{n},\mathbb{R}\right)  $ of bounded Lebesgue measurable functions
$f:\mathbb{R}^{n}\rightarrow\mathbb{R}$ with the supremum norm%
\[
\left\|  f\right\|  _{\infty}:=\underset{x\in\mathbb{R}^{n}}{\sup}\left|
f\left(  x\right)  \right|  \text{.}%
\]
We emphasize that we are not using the essential supremum on equivalence
classes of functions. $\mathcal{B}\left(  \mathbb{R}^{n},\mathbb{R}\right)  $
is a collection of actual functions\footnote{So why is $L^{\infty}$ more
commonly used in analysis? One answer is that $L^{\infty}$ is the dual of
$L^{1}$ and they occupy the ends of the $L^{p}$ spectrum. The elements of
$L^{1}$ require equivalence classes of functions for a clean definition, and,
unlike $L^{\infty}$, doesn't have an analog metric which can distinguish each
function and remain complete.
\par
We have the following inequalities:
\par
1. For $f\in L^{1}$ and $g\in L^{\infty}$ we have
\par
$\left\|  fg\right\|  _{1}\leq\left\|  f\right\|  _{1}$essup$\left(  \left|
g\right|  \right)  $ and
\par
2. For $f\in L^{1}$ and $g\in B\left(  \mathbb{R},\mathbb{R}\right)  $ we have
\par
$\left\|  fg\right\|  _{1}\leq\left\|  f\right\|  _{1}\sup\left(  \left|
g\right|  \right)  $
\par
\noindent However, the first inequality is sharper than the second as we can
see by changing $g$ arbitrarily on a set of measure $0$.
\par
Perhaps more importantly, the existence of the Hilbert space $L^{2}$ in the
middle of the $L^{p}$ spectrum has a basis which makes it crucial in
functional analysis and its applications, such as signal analysis.}, so we can
refer to $f\left(  x\right)  $ when defining $V\left(  f\right)  $. We take
the domain of the vector field $V$ to be the closed subset consisting of
measurable functions bounded by $\left[  0,1\right]  $%
\[
M:=\mathcal{B}\left(  \mathbb{R}^{n},\left[  0,1\right]  \right)  :=\left\{
f:\mathbb{R}^{n}\rightarrow\left[  0,1\right]  \mid f\text{ is Lebesgue
measurable}\right\}
\]
with the supremum metric $d_{\infty}\left(  f,g\right)  :=\left\|
f-g\right\|  _{\infty}$. Then it is straightforward to check $M\subset
\mathcal{B}\left(  \mathbb{R}^{n},\mathbb{R}\right)  $ is closed and therefore
$\left(  M,d_{\infty}\right)  $ is a complete metric space.

However, the greater obstacle to this effort is that the vector field $V$ is
not continuous as the following example shows.

\begin{example}
Analyzing $V$ near the constant $0$ function $f\left(  x\right)  =0$ with the
input $g\left(  x\right)  :=\epsilon1_{\left[  0,1\right]  }\left(  x\right)
$ for some small $\epsilon>0$ shows $V:M\rightarrow M$ is not continuous with
respect to $d_{\infty}$. For a typical growth function $G$ described in Chan's
paper, we have $G\left(  0\right)  =-1=G\left(  1\right)  $ and $G\left(
1/2\right)  =1$. (These numbers can be modified to fit many other choices of
$G$ and the argument still follows.) Then for kernel $K\in L^{1}$ we have
$K\ast f=0\approx K\ast g$ so $G\left(  K\ast f\right)  =-1\approx G\left(
K\ast g\right)  $. But even though $f$ and $g$ are $\epsilon$-close, the
formula for $V$ clips for $f$ but not for $g$, and we have%
\[
\left|  \left|  V\left(  f\right)  -V\left(  g\right)  \right|  \right|
\approx\left|  \left|  0-\left(  -1\right)  \right|  \right|  =1
\]
So the distance between $V\left(  f\right)  $ and $V\left(  g\right)  $ is not
small, even though $f$ and $g$ can be made arbitrarily close. $V$ is not
continuous in the $L^{p}\mathbb{\ }$metric for any $1\leq p\leq\infty$.
\end{example}

\bigskip

Therefore the basic differential equation theory for Lipschitz vector fields
on Banach spaces is not sufficient to guarantee the existence of solutions.
Another approach is required. The natural setting for this model is not vector
fields on linear spaces, but arc fields on metric spaces. In fact, Chan's
formula \cite[(9), p. 256]{Chan1}%
\begin{equation}
f_{t+\Delta t}(x)=\left[  f_{t}(x)+\Delta tG\left(  K\ast f_{t}\right)
(x)\right]  _{0}^{1}\label{LineChanFormAF}%
\end{equation}
is not the traditional Euler approximation for the integral curves to the
vector field $V$ which would instead start as $f_{t+\Delta t}\approx
f_{t}+\Delta tV\left(  f_{t}\right)  $. What Chan is actually doing is using
an arc field $X$ instead of the vector field $V$ for the Euler curve algorithm
to simulate life in Lenia, as explained in the next section.

\section{Lenia arc field}

Lenia lives on the metric space $M:=\mathcal{B}\left(  \mathbb{R}^{n},\left[
0,1\right]  \right)  $ with the supremum metric $d_{\infty}\left(  f,g\right)
:=\left\|  f-g\right\|  _{\infty}$. Chan's instruction $\left(
\ref{LineChanFormAF}\right)  $ for the motion of these creatures $f$ is the
Lenia arc field $X$ which is a map $X:M\times\left[  0,1\right]  \rightarrow
M$ given by%
\begin{equation}
X_{t}\left(  f\right)  =\left[  f+tG(K\ast f)\right]  _{0}^{1}%
\label{LineDefLeniaArcField}%
\end{equation}
Here $t$ represents time, and for each $t\in\left[  0,1\right]  $ the arc
field displays the tendency to change any creature $f\in M$ toward a new
creature $X_{t}\left(  f\right)  \in M$. Notice $X_{0}\left(  f\right)  =f$
which means the arc field starting at time $t=0$ with $f$ gives a
``direction'' for change, using a curve $X_{\cdot}\left(  f\right)  $ in the
metric space $M$.

The arc field $X$ is used to generate a continuous flow $F:M\times\left[
0,\infty\right)  \rightarrow M$ using the familiar Euler curve technique%
\[
F_{t}\left(  f\right)  =\underset{n\rightarrow\infty}{\lim}X_{t/n}^{\left(
n\right)  }\left(  f\right)
\]
where the superscript parentheses denote composition $n$ times:%
\[
X_{t}^{\left(  n\right)  }=\text{ }\underset{n\text{ compositions}%
}{\underbrace{X_{t}\circ X_{t}\circ\cdots\circ X_{t}}}\text{ }=\overset
{n}{\underset{k=1}{\bigcirc}}X_{t}\text{.}%
\]
The main result of this paper is to prove that this Euler curve algorithm
converges to a flow. We also prove the resulting flow is forward tangent to
the directions given by the arc field $X$, that the flow is unique, and the
flow exists for all forward time. These facts are proven in Section 2 after
reviewing the relevant results from arc field theory. Then we show this flow
is the desired solution to the original integro-differential equation.

\section{Arc field theory}

In this subsection we review the results of arc field theory
\cite{CalcBleecker} used to analyze Lenia throughout the rest of the paper.

\begin{definition}
\label{DefArcField}An \textbf{arc field} on a metric space $M$ is a continuous
map $X:M\times\left[  0,1\right]  \rightarrow M$ such that for all $f\in M$,
$X_{0}\left(  f\right)  =f$.

The \textbf{speed} of an arc field at $f\in M$ is
\[
\rho\left(  f\right)  :=\sup_{s\neq t}\frac{d\left(  X_{s}\left(  f\right)
,X_{t}\left(  f\right)  \right)  }{\left|  s-t\right|  }%
\]
i.e., the curve in $M$ starting at $f$ given by $X_{\left(  \cdot\right)
}\left(  f\right)  :\left[  0,1\right]  \rightarrow M$ is Lipschitz with
constant $\rho\left(  f\right)  $. An arc field $X$ has \textbf{linear speed
growth} if there is a point $f\in M$ and positive constants $c_{1}$ and
$c_{2}$ such that for all $r>0$%
\[
\rho\left(  f,r\right)  :=\sup\left\{  \rho\left(  g\right)  \mid g\in
B\left(  f,r\right)  \right\}
\]
satisfies
\begin{equation}
\rho\left(  f,r\right)  \leq c_{1}r+c_{2}\text{.}\label{LinGro}%
\end{equation}
\end{definition}

Notice $M$ is an arbitrary metric space and does not need to be a function
space, though that is the structure Lenia uses in its model.

\begin{definition}
A (forward) \textbf{flow} on a metric space $M$ is a continuous map
$F:M\times\left[  0,\infty\right)  \rightarrow M$ which satisfies for all
$f\in M$

\noindent$\left(  i\right)  \quad F_{0}(f)=f$

\noindent$\left(  ii\right)  $\quad$F_{t}(F_{s}(f))=F_{s+t}(f)$ for all
$s,t\geq0$.

$F$ is a \textbf{flow of an arc field }$X$ if in addition we have $F$ is
\textbf{(forward) tangent} to $X$, meaning%
\begin{equation}
\lim_{h\rightarrow0^{+}}\frac{d\left(  F_{t+h}\left(  f\right)  ,X_{h}\left(
F_{t}\left(  f\right)  \right)  \right)  }{h}=0\label{LineTangencyCondition}%
\end{equation}
for all $t\geq0$ and $f\in M$.
\end{definition}

Due to tangency, each curve from the flow $F_{\left(  \cdot\right)  }\left(
f\right)  :\left[  0,\infty\right)  \rightarrow M$ is the analog of an
integral curve (think of a solution to an evolutionary PDE), just as an arc
field $X$ is the metric space analog of a vector field.

The linear bound on speed in the definition of an arc field $X$ is used to
prove the resulting flow $F$ tangent to $X$ exists for all time. The ODE
$x^{\prime}=x^{2}$ for $x\left(  t\right)  \in M:=\mathbb{R}$ is the standard
example used to show the linear speed growth assumption is natural, since this
ODE has quadratic speed growth and solutions diverge to infinity in finite time.

\begin{proposition}
\label{ThmSpdBd}The Lenia arc field $\left(  \ref{LineDefLeniaArcField}%
\right)  $%
\[
X_{t}\left(  f\right)  \left(  x\right)  =\left[  f\left(  x\right)  +tG(K\ast
f)\left(  x\right)  \right]  _{0}^{1}%
\]
on $M:=\mathcal{B}\left(  \mathbb{R}^{n},\left[  0,1\right]  \right)  $ with
supremum metric $d_{\infty}$ has speed globally bounded by $\max\left|
G\right|  $.
\end{proposition}

\begin{proof}%
\begin{align*}
\rho\left(  f\right)   & =\sup_{s\neq t}\frac{d\left(  X\left(  f,s\right)
,X\left(  f,t\right)  \right)  }{\left|  s-t\right|  }\\
& =\sup_{s\neq t}\frac{\left\|  \left[  f+sG(K\ast f)\right]  _{0}^{1}-\left[
f+tG(K\ast f)\right]  _{0}^{1}\right\|  _{\infty}}{\left|  s-t\right|  }\\
& \leq\sup_{s\neq t}\frac{\left\|  f+sG(K\ast f)-\left(  f+tG(K\ast f)\right)
\right\|  _{\infty}}{\left|  s-t\right|  }\\
& =\sup_{s\neq t}\left\|  G(K\ast f)\right\|  _{\infty}\leq\max\left|
G\right|  \text{.}%
\end{align*}
\end{proof}

The following conditions are used to guarantee arc fields on arbitrary
complete metric spaces generate flows via Euler curves.

\noindent\textbf{Condition E1:} For each $f_{0}\in M$, there are constants
$r>0,$ $\delta>0$ and $\Lambda\in\mathbb{R}$ such that for all $f,g\in
B\left(  f_{0},r\right)  $ and $0\leq t<\delta$%
\[
d\left(  X_{t}\left(  f\right)  ,X_{t}\left(  g\right)  \right)  \leq\left(
1+t\Lambda\right)  d\left(  f,g\right)  \text{.}%
\]

\noindent\textbf{Condition E2}: For each $f_{0}\in M$, there are constants
$r>0,$ $\delta>0$ and $\Omega\in\mathbb{R}$ such that for all $f\in B\left(
f_{0},r\right)  $ and $0\leq t\leq s<\delta$%
\[
d\left(  X_{s+t}\left(  f\right)  ,X_{t}\left(  X_{s}\left(  f\right)
\right)  \right)  =st\Omega\text{.}%
\]

Infinitesimally E1 limits the spread of $X$ in a linear fashion. E2 restrains
$X$ to be almost flow-like, since $X$ would be a flow if it satisfied E2 with
constant $\Omega=0$.

\begin{theorem}
\label{ThmFundAF}Let $X$ be an arc field on a complete metric space $M$ which
satisfies E1 and E2 and has linear speed growth. Then there exists a unique
forward flow $F\ $tangent to $X$.

The Euler curve algorithm converges to this flow, meaning%
\[
F_{t}\left(  f\right)  =\underset{n\rightarrow\infty}{\lim}X_{t/n}^{\left(
n\right)  }\left(  f\right)
\]
where%
\[
X_{t}^{\left(  n\right)  }=\text{ }\underset{n\text{ compositions}%
}{\underbrace{X_{t}\circ X_{t}\circ\cdots\circ X_{t}}}\text{ }=\overset
{n}{\underset{k=1}{\bigcirc}}X_{t}\text{.}%
\]
\end{theorem}

The proof of this fundamental theorem is given in \cite{CalcBleecker}.

When $X$ merely has locally bounded speed instead of linearly bounded speed,
we can only guarantee a local flow, meaning time $t$ may be limited in
different ways before solutions blow up, depending on the initial condition
$f$.

Theorem \ref{ThmFundAF} generalizes the classical Fundamental Theorem of ODEs
(also known as the Picard-Lindel\"{o}f Theorem or the Cauchy-Lipschitz
Theorem) which states that Lipschitz vector fields on Banach spaces have
unique integral curves. One of the results of this paper is to show that
Theorem \ref{ThmFundAF} also guarantees solutions to some interesting
non-continuous vector fields on function spaces.

\section{Clip function properties}

In this section we collect some facts about the clip function that will be
used in the existence proof, Theorem \ref{ThmGen1}, below. The clip function%
\[
\left[  x\right]  _{a}^{b}:=\left\{
\begin{array}
[c]{cc}%
b & \text{if }x\geq b\\
x & \text{if }a<x<b\\
a & \text{if }x\leq a
\end{array}
\right.  =\min\left\{  \max\left\{  a,x\right\}  ,b\right\}
\]
is well-defined whenever $a\leq b$. We will also use the notation
\begin{align*}
\left[  x\right]  _{a} &  :=\max\left\{  x,a\right\} \\
\left[  x\right]  ^{b} &  :=\min\left\{  x,b\right\}
\end{align*}
for the one-sided clips. Notice $\left[  x\right]  _{0}=x\cdot1_{\left[
0,\infty\right)  }\left(  x\right)  $ is a low pass filter where $1_{S}$ is
the indicator function on the set $S$%
\[
1_{s}\left(  x\right)  :=\left\{
\begin{array}
[c]{cc}%
1 & x\in S\\
0 & \text{otherwise.}%
\end{array}
\right.
\]
More generally $\left[  x\right]  _{a}=\left(  x-a\right)  \cdot1_{\left[
0,\infty\right)  }\left(  x\right)  +a$ and $\left[  x\right]  ^{b}=\left(
x-b\right)  \cdot1_{\left(  -\infty,0\right]  }\left(  x\right)  +b$. So
composing them shows $\left[  x\right]  _{a}^{b}=\left[  \left[  x\right]
^{b}\right]  _{a}=\left[  \left[  x\right]  _{a}\right]  ^{b}$ is a band pass filter.

Let $a\leq b$ and $c\leq d$ and $x,y\in\mathbb{R}$. It is straightforward to
verify the following identities and inequalities

\begin{itemize}
\item $\left[  rx\right]  _{a}^{b}=r\left[  x\right]  _{a/r}^{b/r}$ for $r>0$
and $\left[  rx\right]  _{a}^{b}=r\left[  x\right]  _{b/r}^{a/r}$ for $r<0$

\item $\left[  x+y\right]  _{a}^{b}=\left[  x\right]  _{a-y}^{b-y}+y$

\qquad$\left[  x+y\right]  ^{b}=\left[  x\right]  ^{b-y}+y$

$\qquad\left[  x+y\right]  _{a}=\left[  x\right]  _{a-y}+y$

\item $\left[  x\right]  _{a}^{b}=\left[  x\right]  ^{b}-\left[  x\right]
^{a}+a$ which gives

$\left[  x\right]  _{a}^{b}-\left[  x\right]  _{c}^{d}=\left[  x\right]
_{d}^{b}-\left[  x\right]  _{c}^{a}-\left(  d-a\right)  $ assuming $b\geq d$
and $a\geq c$

\item $\left[  \left[  x\right]  _{a}^{b}\right]  _{c}^{d}=\left[  x\right]
_{\max\left\{  a,c\right\}  }^{\min\left\{  b,d\right\}  }$ when $\max\left\{
a,c\right\}  \leq\min\left\{  b,d\right\}  $
\end{itemize}

\bigskip

\begin{itemize}
\item $\left|  \left[  x\right]  _{a}^{b}-\left[  y\right]  _{a}^{b}\right|
\leq\left|  x-y\right|  $

i.e., the clip function $x\longmapsto\left[  x\right]  _{a}^{b}$ is Lipschitz
with constant $1$.

\qquad$\left|  \left[  x\right]  ^{b}-\left[  y\right]  ^{b}\right|
\leq\left|  x-y\right|  $

\qquad$\left|  \left[  x\right]  _{a}-\left[  y\right]  _{a}\right|
\leq\left|  x-y\right|  $

\item $\left|  \left[  x\right]  _{a}^{b}-\left[  x\right]  _{c}^{d}\right|
\leq\max\left\{  \left|  a-c\right|  ,\left|  b-d\right|  \right\}  $%
\newline and generalizing the two previous facts, we have

\item $\left|  \left[  x\right]  _{a}^{b}-\left[  y\right]  _{c}^{d}\right|
\leq\max\left\{  \left|  x-y\right|  ,\left|  a-c\right|  ,\left|  b-d\right|
\right\}  $

\item $0\leq\left[  x\right]  _{b}^{c}-\left[  x\right]  _{a}^{b}\leq c-a$
assuming $a\leq b\leq c$

\item $\left|  \left[  x\right]  _{a}^{b}\right|  \leq\left|  x\right|  $
whenever $a\leq0\leq b$
\end{itemize}

\section{Existence and Uniqueness of Solutions to the Lenia Arc Field}

We will show the Lenia arc field $\left(  \ref{LineDefLeniaArcField}\right)  $
satisfies the regularity conditions, E1 and E2, which guarantee the existence
and uniqueness of a forward flow tangent to the arc field by Theorem
\ref{ThmFundAF}. However the calculations get a bit complicated. As a warmup
for using the clip function on the full arc field, and to get a feel for how
arc fields can handle discontinuities in directions, we recommend first
working through the following simpler example which is not used in the
remainder of the paper.

\begin{example}
\label{ExToyDisc}Let $V:\mathbb{R}\rightarrow\mathbb{R}$ be the discontinuous
vector field given by%
\[
V\left(  x\right)  :=\left\{
\begin{array}
[c]{cc}%
-1 & \text{if }x>0\\
0 & \text{if }x\leq0\text{.}%
\end{array}
\right.
\]
$V\left(  x\right)  $ gives the direction for $x\in M:=\mathbb{R}$. Even
though the directions from $V$ are clear, and the resulting dynamic has a
unique formula for all time, because $V$ is discontinuous the traditional
theory (Picard-Lindel\"{o}f Theorem) does not apply.

To make an analog arc field to $V$ we replace the vectors with curves in the
same direction. Therefore we mechanically choose $X:M\times\left[  0,1\right]
\rightarrow M$ with $M:=\mathbb{R}$ defined by%
\[
X_{t}\left(  x\right)  :=\left[  x-t\right]  _{0}%
\]
where the square brackets again denote the clip function. This arc field has
the dynamic that solutions with initial conditions $x>0$ move down until
hitting $0$, then stop. This immediate stop displays a discontinuous direction
change, but the arc field is still continuous in $t$ for each $x$.

Let's check E1 using $\left|  \left[  x\right]  _{0}-\left[  y\right]
_{0}\right|  \leq\left|  x-y\right|  $.%
\begin{align*}
d\left(  X_{t}\left(  x\right)  ,X_{t}\left(  y\right)  \right)   &
\leq\left|  \left[  x-t\right]  _{0}-\left[  y-t\right]  _{0}\right| \\
& \leq\left|  \left(  x-t\right)  -\left(  y-t\right)  \right|  =d\left(
x,y\right)  \text{.}%
\end{align*}
so $\Lambda=0$ is sufficient.

Then E2 is calculated using the property of the clip $\left[  x+y\right]
_{0}=\left[  x\right]  _{-y}+y$%
\begin{align*}
d\left(  X_{s+t}\left(  x\right)  ,X_{t}\left(  X_{s}\left(  x\right)
\right)  \right)   & =\left|  \left[  x-\left(  s+t\right)  \right]
_{0}-\left[  \left[  x-s\right]  _{0}-t\right]  _{0}\right| \\
& =\left|  \left[  -\left(  s+t\right)  \right]  _{-x}+x-\left(  \left[
\left[  -s\right]  _{-x}+x-t\right]  _{0}\right)  \right| \\
& =\left|  \left[  -\left(  s+t\right)  \right]  _{-x}+x-\left(  \left[
\left[  -s\right]  _{-x}\right]  _{-x+t}+x-t\right)  \right| \\
& =\left|  \left[  -\left(  s+t\right)  \right]  _{-x}-\left(  \left[
-s\right]  _{-x+t}-t\right)  \right| \\
& =\left|  \left[  -\left(  s+t\right)  \right]  _{-x}-\left(  \left[
-s-t\right]  _{-x}\right)  \right|  =0
\end{align*}
where the third to fourth line calculation used $\left[  x\right]  _{a}%
:=\max\left\{  a,x\right\}  $ so that $\left[  \left[  -s\right]
_{-x}\right]  _{-x+t}=\left[  -s\right]  _{-x+t}$ for $t\geq0$.

Since the speed of $X$ is globally bounded by $1$ on the complete metric space
$M$, Theorem \ref{ThmFundAF} guarantees the Euler curves generate a unique
forward flow tangent to $X$. Notice $F=X$ in this trivial scenario.
\end{example}

\textit{N.b.}: Precisely how uniqueness of solutions backward in time fails is
quite complicated for the Lenia dynamic. Considering the reverse dynamic of
the previous example is illuminating for understanding the issue. In
particular notice that two different initial conditions can coincide after a
finite amount of time. For instance $F_{t}\left(  x\right)  =0$ for all $x\in
M$ for any $t\geq\left|  x\right|  $. The forward flow is still unique, but
the backward flow is not well-defined from $x=0$ since there are multiple
inverses to choose from. This doesn't happen with Lipschitz continuous vector
fields, where forward and backward flows exist uniquely, so integral curves
cannot intersect at any finite time $t$. Similarly in the full Lenia model,
depending on the choice of $K$ and $G$, it is possible for many different
initial conditions $f$ to go extinct, i.e., $F_{t}\left(  f\right)  =0$ after
finite time $t$.

\bigskip

We next prove the Lenia arc field $X$ has a flow using similar calculations
from the above toy example, Example \ref{ExToyDisc}. Now, however, there are
further complications, because the number of variables to consider is infinite
instead of 1. Since we are taking the supremum over all possible
$x\in\mathbb{R}^{n}$ we have to account for all possible cases. The clip
function%
\[
\left[  x\right]  _{a}^{b}=\left\{
\begin{array}
[c]{cc}%
b & \text{if }x\geq b\\
x & \text{if }a<x<b\\
a & \text{if }x\leq a
\end{array}
\right.
\]
naturally breaks into 3 cases each time it's used. When we check E2 we need to
consider $X_{t}\left(  X_{s}\right)  $ which means we need to deal with the
clip of the clip, so the cases multiply, and the proof is inevitably complicated.

The major new result of the paper is the following theorem.

\begin{theorem}
\label{ThmExUn} Consider the complete metric space $M:=\mathcal{B}\left(
\mathbb{R}^{n},\left[  0,1\right]  \right)  $ with the supremum metric
$d_{\infty}$. Let $K\in L^{1}\ $meaning $K:\mathbb{R}^{n}\rightarrow
\mathbb{R}$ satisfies $\int_{\mathbb{R}^{n}}\left|  K\left(  x\right)
\right|  dx<\infty$. Let $G:\mathbb{R}\rightarrow\mathbb{R}$ be bounded and
Lipschitz continuous. Then the Lenia arc field $X:M\times\left[  0,1\right]
\rightarrow M$ given by%
\[
X_{t}\left(  f\right)  \left(  x\right)  :=\left[  f\left(  x\right)
+tG(K\ast f)\left(  x\right)  \right]  _{0}^{1}%
\]
satisfies Conditions E1 and E2.

Thus $X$ generates a unique forward flow $F:M\times\left[  0,\infty\right)
\rightarrow M$ forward tangent to $X$.

$F$ is approximated by Euler curves%
\[
F_{t}\left(  f\right)  =\underset{n\rightarrow\infty}{\lim}X_{t/n}^{\left(
n\right)  }\left(  f\right)  \text{.}%
\]
\end{theorem}

\begin{proof}
This is a consequence of the more general Theorem \ref{ThmGen1}, proven next.
Specifically, the space of measurable functions $\mathcal{B}\left(
\mathbb{R}^{n},\left[  0,1\right]  \right)  $ is a closed subset of
$\mathcal{B}\left(  \mathbb{R}^{n},\mathbb{R}\right)  $ under the $\sup$ norm.
Then remember the general fact from Fourier Theory about convolutions that if
$K\in L^{1}\left(  \mathbb{R}^{n}\right)  $ and $f\in L^{p}\left(
\mathbb{R}^{n}\right)  $ for any $1\leq p\leq\infty$ then $K\ast f\in
L^{p}\left(  \mathbb{R}^{n}\right)  $ and $\left\|  K\ast f\right\|
_{p}=\left\|  K\right\|  _{1}\left\|  f\right\|  _{p}$ by Fubini's Theorem.
Since $G$ is bounded we see $X$ is continuous in $t$ for each $f\in M$.
Further $X_{t}\left(  f\right)  $ is continuous in $f$ since $G$ is, and $X$
is a well-defined arc field on $M$. Finally, the Lenia arc field is
$X_{t}\left(  f\right)  :=\left[  f+tV\left(  f\right)  \right]  _{0}^{1}$
where $V\left(  f\right)  :=G\left(  K\ast f\right)  $ and this $V$ is
Lipschitz with constant $C_{G}\left\|  K\right\|  _{1}$ where $C_{G}$ is the
Lipschitz constant of $G$ since%
\begin{gather*}
\left\|  V\left(  f\right)  -V\left(  g\right)  \right\|  _{\infty}=\left\|
G\left(  K\ast f\right)  -G\left(  K\ast g\right)  \right\|  _{\infty}\\
\leq C_{G}\left\|  K\ast\left(  f-g\right)  \right\|  _{\infty}=C_{G}\left\|
K\right\|  _{1}\left\|  f-g\right\|  _{\infty}%
\end{gather*}
\end{proof}

\begin{theorem}
\label{ThmGen1}Let $S$ be a measure space and let $a,b:S\rightarrow\left[
-\infty,\infty\right]  $ be functions on $S$ separated by $m:=\underset{x\in
S}{\inf}\left(  b\left(  x\right)  -a\left(  x\right)  \right)  >0$. Let
$\mathcal{B}:=\mathcal{B}\left(  S,\mathbb{R}\right)  $ be the Banach space of
bounded measurable functions $f:S\rightarrow\mathbb{R}$ with norm $\left\|
f\right\|  _{\infty}:=\underset{x\in S}{\sup}\left|  f\left(  x\right)
\right|  $. Define the metric space subspace of $\mathcal{B}$%
\[
M:=\left\{  f\in\mathcal{B}\mid a\left(  x\right)  \leq f\left(  x\right)
\leq b\left(  x\right)  \text{ }\forall x\in S\right\}
\]
with metric given by $d_{\infty}\left(  f,g\right)  :=\left\|  f-g\right\|
_{\infty}$. Finally, let $V:M\rightarrow\mathcal{B}\left(  S,\mathbb{R}%
\right)  $ be a locally Lipschitz map.

Then the arc field $X:M\times\left[  0,1\right]  \rightarrow M$ given by%
\[
X_{t}\left(  f\right)  \left(  x\right)  :=\left[  f+tV\left(  f\right)
\right]  _{a}^{b}\left(  x\right)  :=\left[  f\left(  x\right)  +tV\left(
f\right)  \left(  x\right)  \right]  _{a\left(  x\right)  }^{b\left(
x\right)  }%
\]
satisfies Conditions E1 and E2 and generates a unique local forward flow $F$
tangent to $X$.

$F$ may be approximated by Euler curves%
\[
F_{t}\left(  f\right)  =\underset{n\rightarrow\infty}{\lim}X_{t/n}^{\left(
n\right)  }\left(  f\right)  \text{.}%
\]

If $V$ is globally Lipschitz, then the flow $F$ exists for all time
$F:M\times\left[  0,\infty\right)  \rightarrow M$.
\end{theorem}

\begin{proof}
Notice $M$ is not necessarily a Banach space since scalar multiples of
elements in $M$ may leave $M$, but it is a complete metric space since it is a
closed subspace $M\subset\mathcal{B}\left(  S,\mathbb{R}\right)  $ of the
Banach space. So $X$ is a well-defined arc field on $M$ and the speed $\rho$
of $X$ is bounded by the norm of the vector field%
\begin{align*}
\rho\left(  f\right)    & :=\sup_{s\neq t}\frac{d_{\infty}\left(  X_{s}\left(
f\right)  ,X_{t}\left(  f\right)  \right)  }{\left|  s-t\right|  }\\
& =\sup_{s\neq t}\frac{\left\|  \left[  f+sV\left(  f\right)  \right]
_{a}^{b}-\left[  f+tV\left(  f\right)  \right]  _{a}^{b}\right\|  _{\infty}%
}{\left|  s-t\right|  }\\
& \leq\sup_{s\neq t}\frac{\left\|  f+sV\left(  f\right)  -\left(  f+tV\left(
f\right)  \right)  \right\|  _{\infty}}{\left|  s-t\right|  }=\left\|
V\left(  f\right)  \right\|  _{\infty}%
\end{align*}
When $V$ is globally Lipschitz with constant $C_{V}$, then $X$ has linear
speed growth since for any $f_{0}\in M$ and $r>0$%
\[
\rho\left(  f\right)  \leq\left\|  V\left(  f\right)  \right\|  _{\infty}%
\leq\left\|  V\left(  f\right)  -V\left(  f_{0}\right)  \right\|  _{\infty
}+\left\|  V\left(  f_{0}\right)  \right\|  _{\infty}%
\]
gives%
\[
\rho\left(  f_{0},r\right)  :=\sup\left\{  \rho\left(  f\right)  |f\in
B\left(  f_{0},r\right)  \right\}  \leq C_{V}r+\left\|  V\left(  f_{0}\right)
\right\|  _{\infty}%
\]

Let $f_{0}\in M$ and let $\epsilon>0$ be such that $V$ has local Lipschitz
constant $C_{V}$ meaning $\left\|  V\left(  f\right)  -V\left(  g\right)
\right\|  _{\infty}\leq C_{V}\left\|  f-g\right\|  _{\infty}$ for all $f,g\in
B\left(  f_{0},\epsilon\right)  $.

E1 follows from%
\begin{align*}
d_{\infty}\left(  X_{t}\left(  f\right)  ,X_{t}\left(  g\right)  \right)    &
=\left\|  \left[  f+tV\left(  f\right)  \right]  _{a}^{b}-\left[  g+tV\left(
g\right)  \right]  _{a}^{b}\right\|  _{\infty}\\
& \leq\left\|  f+tV\left(  f\right)  -\left(  g+tV\left(  g\right)  \right)
\right\|  _{\infty}\\
& \leq\left\|  f-g\right\|  _{\infty}+t\left\|  V\left(  f\right)  -V\left(
g\right)  \right\|  _{\infty}\\
& \leq\left\|  f-g\right\|  _{\infty}+tC_{V}\left\|  f-g\right\|  _{\infty}\\
& =d_{\infty}\left(  f,g\right)  \left(  1+tC_{V}\right)
\end{align*}
So the local $\Lambda$ for E1 is the local Lipschitz constant $C_{V}$ of $V$.

The last thing to check is E2, which is significantly more complicated. Now
let $s,t$ be such that $0\leq t\leq s$ satisfy $s\leq\min\left\{
\frac{\epsilon}{3\rho},\frac{m}{\rho}\right\}  $ where $m:=\underset{x\in
S}{\inf}\left(  b\left(  x\right)  -a\left(  x\right)  \right)  >0$ so that
$f+sV\left(  f\right)  \in B\left(  f_{0},\epsilon\right)  $ for any $f\in
B\left(  f_{0},\epsilon/3\right)  $. It will be more obvious later that this
is sufficient restriction to guarantee $X_{t}X_{s}\left(  f\right)  $ stays in
$B\left(  f_{0},\epsilon\right)  $ and the subsequent calculations can count
on the spread of $a<b$ to remain valid.

Using the formula $\left[  x+y\right]  _{a}^{b}=\left[  x\right]  _{a-y}%
^{b-y}+y$ we calculate%
\begin{align}
& d_{\infty}\left(  X_{s+t}\left(  f\right)  ,X_{t}\left(  X_{s}\left(
f\right)  \right)  \right)  =\left\|  X_{s+t}\left(  f\right)  -X_{t}\left(
X_{s}\left(  f\right)  \right)  \right\|  _{\infty}\nonumber\\
& =\left\|
\begin{array}
[c]{c}%
\left[  f+\left(  s+t\right)  V\left(  f\right)  \right]  _{a}^{b}\\
-\left[  \left[  f+sV\left(  f\right)  \right]  _{a}^{b}+tV\left(  \left[
f+sV\left(  f\right)  \right]  _{a}^{b}\right)  \right]  _{a}^{b}%
\end{array}
\right\|  _{\infty}\nonumber\\
& =\left\|
\begin{array}
[c]{c}%
\left[  f+sV\left(  f\right)  \right]  _{a-tV\left(  f\right)  }^{b-tV\left(
f\right)  }+tV\left(  f\right)  \\
-\left[  \left[  f+sV\left(  f\right)  \right]  _{a}^{b}\right]  _{a-tV\left(
\left[  f+sV\left(  f\right)  \right]  _{a}^{b}\right)  }^{b-tV\left(  \left[
f+sV\left(  f\right)  \right]  _{a}^{b}\right)  }-tV\left(  \left[
f+sV\left(  f\right)  \right]  _{a}^{b}\right)
\end{array}
\right\|  _{\infty}\nonumber\\
& \leq%
\begin{array}
[c]{c}%
\left\|
\begin{array}
[c]{c}%
\left[  f+sV\left(  f\right)  \right]  _{a-tV\left(  f\right)  }^{b-tV\left(
f\right)  }\\
-\left[  \left[  f+sV\left(  f\right)  \right]  _{a}^{b}\right]  _{a-tV\left(
\left[  f+sV\left(  f\right)  \right]  _{a}^{b}\right)  }^{b-tV\left(  \left[
f+sV\left(  f\right)  \right]  _{a}^{b}\right)  }%
\end{array}
\right\|  _{\infty}\\
+t\left\|  V\left(  f\right)  -V\left(  \left[  f+sV\left(  f\right)  \right]
_{a}^{b}\right)  \right\|  _{\infty}%
\end{array}
\nonumber\\
& \leq\left\|
\begin{array}
[c]{c}%
\left[  f+sV\left(  f\right)  \right]  _{a-tV\left(  f\right)  }^{b-tV\left(
f\right)  }\\
-\left[  \left[  f+sV\left(  f\right)  \right]  _{a}^{b}\right]  _{a-tV\left(
\left[  f+sV\left(  f\right)  \right]  _{a}^{b}\right)  }^{b-tV\left(  \left[
f+sV\left(  f\right)  \right]  _{a}^{b}\right)  }%
\end{array}
\right\|  _{\infty}+\left(  st\right)  C_{V}\rho\label{LineCalc210}%
\end{align}
where $\rho:=\rho\left(  f_{0},\epsilon\right)  \ $which is the finite speed
of $X$. This last line follows because the bound on $s$ keeps $X_{s}\left(
f\right)  $ inside $B\left(  f_{0},2\epsilon/3\right)  $ for any $f\in
B\left(  f_{0},\epsilon/3\right)  $ by the definition of $\rho$. Thus the same
local Lipschitz constant $C_{V}$ and the bound on the speed $\rho$ of $X$
holds so
\begin{align}
& \left\|  V\left(  f\right)  -V\left(  \left[  f+sV\left(  f\right)  \right]
_{a}^{b}\right)  \right\|  _{\infty}\nonumber\\
& \leq C_{V}\left\|  f-\left[  f+sV\left(  f\right)  \right]  _{a}%
^{b}\right\|  _{\infty}\nonumber\\
& =C_{V}\left\|  X_{0}\left(  f\right)  -X_{s}\left(  f\right)  \right\|
_{\infty}\leq C_{V}s\rho\text{.}\label{LineCalc220}%
\end{align}
We will use this calculation twice more in this proof without comment noticing
$X_{t}X_{s}\left(  f\right)  $ also stays inside the ball $B\left(
f_{0},\epsilon\right)  $ since $t<s$.

So the second term $\left(  st\right)  C_{V}\rho=O\left(  st\right)  $ of line
$\left(  \ref{LineCalc210}\right)  $ works for the E2 condition.

Now concentrating on similarly bounding the first term of $\left(
\ref{LineCalc210}\right)  $%
\[
\left\|
\begin{array}
[c]{c}%
\left[  f+sV\left(  f\right)  \right]  _{a-tV\left(  f\right)  }^{b-tV\left(
f\right)  }\\
-\left[  \left[  f+sV\left(  f\right)  \right]  _{a}^{b}\right]  _{a-tV\left(
\left[  f+sV\left(  f\right)  \right]  _{a}^{b}\right)  }^{b-tV\left(  \left[
f+sV\left(  f\right)  \right]  _{a}^{b}\right)  }%
\end{array}
\right\|  _{\infty}%
\]
we consider the cases when $-tV\left(  \left[  f+sV\left(  f\right)  \right]
_{a}^{b}\right)  \left(  x\right)  $ is greater and smaller than $0$.

Remembering $a\leq f\leq b$ and using the general fact of the clip that
$\left[  \left[  x\right]  _{a}^{b}\right]  _{c}^{d}=\left[  x\right]
_{\max\left\{  a,c\right\}  }^{\min\left\{  b,d\right\}  }$ when $\max\left\{
a,c\right\}  \leq\min\left\{  b,d\right\}  $ we have four cases to check:%
\begin{align*}
& \left[  \left[  f+sV\left(  f\right)  \right]  _{a}^{b}\right]
_{a-tV\left(  \left[  f+sV\left(  f\right)  \right]  _{a}^{b}\right)
}^{b-tV\left(  \left[  f+sV\left(  f\right)  \right]  _{a}^{b}\right)
}\left(  x\right)  \\
& =\left\{
\begin{array}
[c]{cc}%
\left[  f+sV\left(  f\right)  \right]  ^{b-tV\left(  \left[  f+sV\left(
f\right)  \right]  _{a}^{b}\right)  }\left(  x\right)   &
\begin{array}
[c]{c}%
\text{if }V\left(  \left[  f+sV\left(  f\right)  \right]  _{a}^{b}\right)
\left(  x\right)  \geq0\\
\text{and }V\left(  f\right)  \left(  x\right)  \geq0
\end{array}
I\\
\left[  f+sV\left(  f\right)  \right]  _{a-tV\left(  \left[  f+sV\left(
f\right)  \right]  _{a}^{b}\right)  }\left(  x\right)   &
\begin{array}
[c]{c}%
\text{if }V\left(  \left[  f+sV\left(  f\right)  \right]  _{a}^{b}\right)
\left(  x\right)  \leq0\\
\text{and }V\left(  f\right)  \left(  x\right)  \leq0
\end{array}
II\\
\left[  f+sV\left(  f\right)  \right]  _{a}^{b-tV\left(  \left[  f+sV\left(
f\right)  \right]  _{a}^{b}\right)  }\left(  x\right)   &
\begin{array}
[c]{c}%
\text{if }V\left(  \left[  f+sV\left(  f\right)  \right]  _{a}^{b}\right)
\left(  x\right)  \geq0\\
\text{and }V\left(  f\right)  \left(  x\right)  \leq0
\end{array}
III\\
\left[  f+sV\left(  f\right)  \right]  _{a-tV\left(  \left[  f+sV\left(
f\right)  \right]  _{a}^{b}\right)  }^{b}\left(  x\right)   &
\begin{array}
[c]{c}%
\text{if }V\left(  \left[  f+sV\left(  f\right)  \right]  _{a}^{b}\right)
\left(  x\right)  \leq0\\
\text{and }V\left(  f\right)  \left(  x\right)  \geq0
\end{array}
IV
\end{array}
\right.
\end{align*}
Notice we are using the previously mentioned bound $t\leq\frac{m}{\rho}$ which
ensures that $\left\|  X_{t}X_{s}\left(  f\right)  -X_{s}\left(  f\right)
\right\|  _{\infty}\leq t\rho\leq b-a$ so $a<b-tV\left(  \left[  f+sV\left(
f\right)  \right]  _{a}^{b}\right)  $ in case $III$ and $a-tV\left(  \left[
f+sV\left(  f\right)  \right]  _{a}^{b}\right)  <b$ in case $IV$ so the clips
are well-defined. 

The first two cases will be handled in a similar way and are relatively easy,
because the dynamics of $X_{s}$ and $X_{t}$ go in the same direction at those
values of $x$. The last two cases give a new twist.

Case I is $V\left(  \left[  f+sV\left(  f\right)  \right]  _{a}^{b}\right)
\left(  x\right)  \geq0$ and $V\left(  f\right)  \left(  x\right)  \geq0$
which gives%
\begin{align*}
& \left|
\begin{array}
[c]{c}%
\left[  f+sV\left(  f\right)  \right]  _{a-tV\left(  f\right)  }^{b-tV\left(
f\right)  }\\
-\left[  \left[  f+sV\left(  f\right)  \right]  _{a}^{b}\right]  _{a-tV\left(
\left[  f+sV\left(  f\right)  \right]  _{a}^{b}\right)  }^{b-tV\left(  \left[
f+sV\left(  f\right)  \right]  _{a}^{b}\right)  }%
\end{array}
\left(  x\right)  \right| \\
& =\left|
\begin{array}
[c]{c}%
\left[  f+sV\left(  f\right)  \right]  ^{b-tV\left(  f\right)  }\\
-\left[  f+sV\left(  f\right)  \right]  ^{b-tV\left(  \left[  f+sV\left(
f\right)  \right]  _{a}^{b}\right)  }%
\end{array}
\left(  x\right)  \right| \\
& \leq\left|
\begin{array}
[c]{c}%
b-tV\left(  f\right) \\
-\left(  b-tV\left(  \left[  f+sV\left(  f\right)  \right]  _{a}^{b}\right)
\right)
\end{array}
\left(  x\right)  \right|  =O\left(  st\right)
\end{align*}
where the last line used $\left|  \left[  x\right]  ^{v}-\left[  x\right]
^{w}\right|  \leq\left|  v-w\right|  $ and then the same calculation as line
$\left(  \text{\ref{LineCalc220}}\right)  $.

For Case II we get an analogous calculation with the same result.

For Case III, we assume $V\left(  \left[  f+sV\left(  f\right)  \right]
_{a}^{b}\right)  \left(  x\right)  \geq0$ and $V\left(  f\right)  \left(
x\right)  \leq0$. We use the general formulas $\left[  x+y\right]  _{a}%
^{b}=\left[  x\right]  _{a-y}^{b-y}+y$ again and $\left[  x\right]  _{a}%
^{b}-\left[  x\right]  _{c}^{d}=\left[  x\right]  _{d}^{b}-\left[  x\right]
_{c}^{a}-\left(  d-a\right)  $ assuming $b\geq d\geq c$ and $b\geq a\geq c$
which will be true when we use it in the third line below.%
\begin{align*}
& \left|
\begin{array}
[c]{c}%
\left[  f+sV\left(  f\right)  \right]  _{a-tV\left(  f\right)  }^{b-tV\left(
f\right)  }\\
-\left[  \left[  f+sV\left(  f\right)  \right]  _{a}^{b}\right]  _{a-tV\left(
\left[  f+sV\left(  f\right)  \right]  _{a}^{b}\right)  }^{b-tV\left(  \left[
f+sV\left(  f\right)  \right]  _{a}^{b}\right)  }%
\end{array}
\left(  x\right)  \right|  \\
& =\left|
\begin{array}
[c]{c}%
\left[  f+sV\left(  f\right)  \right]  _{a-tV\left(  f\right)  }^{b-tV\left(
f\right)  }\\
-\left[  f+sV\left(  f\right)  \right]  _{a}^{b-tV\left(  \left[  f+sV\left(
f\right)  \right]  _{a}^{b}\right)  }%
\end{array}
\left(  x\right)  \right|  \\
& =\left|
\begin{array}
[c]{c}%
\left[  f+sV\left(  f\right)  \right]  _{b-tV\left(  \left[  f+sV\left(
f\right)  \right]  _{a}^{b}\right)  }^{b-tV\left(  f\right)  }\\
-\left[  f+sV\left(  f\right)  \right]  _{a}^{a-tV\left(  f\right)  }%
\end{array}
\left(  x\right)  -\left(
\begin{array}
[c]{c}%
b-tV\left(  \left[  f+sV\left(  f\right)  \right]  _{a}^{b}\right)  \\
-\left(  a-tV\left(  f\right)  \right)
\end{array}
\left(  x\right)  \right)  \right|  \\
& =\left|
\begin{array}
[c]{c}%
\left[  f+sV\left(  f\right)  -\left(  b-a\right)  \right]  _{a-tV\left(
\left[  f+sV\left(  f\right)  \right]  _{a}^{b}\right)  }^{a-tV\left(
f\right)  }\\
-\left[  f+sV\left(  f\right)  \right]  _{a}^{a-tV\left(  f\right)  }%
\end{array}
\left(  x\right)  -\left(
\begin{array}
[c]{c}%
-tV\left(  \left[  f+sV\left(  f\right)  \right]  _{a}^{b}\right)  \\
-\left(  -tV\left(  f\right)  \right)
\end{array}
\left(  x\right)  \right)  \right|  \\
& \leq\left|
\begin{array}
[c]{c}%
\left[  f+sV\left(  f\right)  -\left(  b-a\right)  \right]  _{a-tV\left(
\left[  f+sV\left(  f\right)  \right]  _{a}^{b}\right)  }^{a-tV\left(
f\right)  }\\
-\left[  f+sV\left(  f\right)  \right]  _{a}^{a-tV\left(  f\right)  }%
\end{array}
\left(  x\right)  \right|  +O\left(  st\right)  \\
& =\left|
\begin{array}
[c]{c}%
\left[  f+sV\left(  f\right)  -\left(  b-a\right)  \right]  _{a-tV\left(
\left[  f+sV\left(  f\right)  \right]  _{a}^{b}\right)  }^{a}\\
-\left[  f+sV\left(  f\right)  \right]  _{a}^{a-tV\left(  f\right)  }%
\end{array}
\left(  x\right)  \right|  +O\left(  st\right)  \\
& \leq\left|  a-tV\left(  \left[  f+sV\left(  f\right)  \right]  _{a}%
^{b}\right)  -\left(  a-tV\left(  f\right)  \right)  \right|  \left(
x\right)  +O\left(  st\right)  \leq O\left(  st\right)  +O\left(  st\right)
=O\left(  st\right)
\end{align*}
where we used $b-tV\left(  \left[  f+sV\left(  f\right)  \right]  _{a}%
^{b}\right)  \geq a$ in the second line by the choice of $t\leq s\leq
\frac{m}{\rho}$ at the beginning of the proof, and $\left[  f+sV\left(
f\right)  -\left(  b-a\right)  \right]  _{a-tV\left(  \left[  f+sV\left(
f\right)  \right]  _{a}^{b}\right)  }^{a-tV\left(  f\right)  }=\left[
f+sV\left(  f\right)  -\left(  b-a\right)  \right]  _{a-tV\left(  \left[
f+sV\left(  f\right)  \right]  _{a}^{b}\right)  }^{a}$ in the fifth to sixth
line because $f-b+sV\left(  f\right)  +a\leq sV\left(  f\right)  +a\leq a$.

The last line uses the trick where $\left|  \left[  x\right]  _{b}^{c}-\left[
y\right]  _{a}^{b}\right|  \leq\left|  c-a\right|  $ for $a\leq b\leq c$.

Case IV is an analogous calculation to Case III and gives the same result.
Since all four cases satisfy the $O\left(  st\right)  $ bound, E2 is
satisfied. Therefore the Euler curves converge to a local flow tangent to $X$.

Under the global Lipschitz condition, the bound on the speed of $X$ is linear
so solutions exist for all time.
\end{proof}

Theorem \ref{ThmGen1} is easily generalized to functions $f:S\rightarrow
\mathbb{R}^{I}$ for arbitrary indexing sets $I$ by thinking of these more
complicated functions as simply $f\in\mathcal{B}\left(  S,\mathbb{R}%
^{I}\right)  \simeq\mathcal{B}\left(  S\times I,\mathbb{R}\right)  $ in the
following corollary, which is used in the next section to give extensions of
the Lenia model on $\mathcal{B}\left(  S,\mathbb{R}^{3}\right)  $.

\begin{corollary}
\label{ThmGen2}Let $S$ and $I$ be measure spaces, and define the Banach space
$\mathcal{B}:=\mathcal{B}\left(  S\times I,\mathbb{R}\right)  $ of bounded
measurable functions $f:S\times I\rightarrow\mathbb{R}$ with $\infty$ norm
$\left\|  f\right\|  _{\infty}:=\underset{x\in S,\text{ }i\in I}{\sup}\left|
f_{i}\left(  x\right)  \right|  $ where we write $f_{i}\left(  x\right)
=f\left(  x,i\right)  $. Define the clip function $\left[  \cdot\right]
_{a}^{b}:\mathcal{B}\rightarrow\mathcal{B}$ by clipping coefficients
independently, i.e., $\left[  f\right]  _{a}^{b}$ is defined for each $x$ and
$i$ as%
\[
\left(  \left(  \left[  f\right]  _{a}^{b}\right)  \left(  x\right)  \right)
_{i}:=\left[  f_{i}\left(  x\right)  \right]  _{a_{i}\left(  x\right)
}^{b_{i}\left(  x\right)  }%
\]
where $a_{i}$ and $b_{i}$ are functions from $S$ to the extended reals
$\left[  -\infty,\infty\right]  $ such that $-\infty\leq a_{i}\left(
x\right)  <b_{i}\left(  x\right)  \leq\infty$ for all $i\in I$ with
$\underset{i\in I\text{, }x\in S}{\inf}\left(  b_{i}\left(  x\right)
-a_{i}\left(  x\right)  \right)  >0$.

Define the metric space $M$ to be the functions bounded by the clip%
\[
M:=\left\{  f\in\mathcal{B}\mid\left[  f\right]  _{a}^{b}=f\right\}
\]
Notice $M$ is simply the image of the clip, and so has metric derived from the
norm on $\mathcal{B}$ given by $d_{\infty}\left(  f,g\right)  :=\left\|
f-g\right\|  _{\infty}$ making $M$ a complete metric space.

Let $V:M\rightarrow M$ be a locally Lipschitz map. Then the arc field
$X:M\times\left[  0,1\right]  \rightarrow M$ given by%
\[
X_{t}\left(  f\right)  :=\left[  f+tV\left(  f\right)  \right]  _{a}^{b}%
\]
satisfies Conditions E1 and E2.

Thus $X$ generates a unique local forward flow $F$ tangent to $X$. When $V$ is
globally Lipschitz the flow $F$ exists for all time $F:M\times\left[
0,\infty\right)  \rightarrow M$.
\end{corollary}

\section{Solution to the integro-differential equation}

In the previous section, we interpreted Chan's formulation of continuous Lenia
as specifying an arc field on the space $M:=\mathcal{B}\left(  \mathbb{R}%
^{n},\left[  0,1\right]  \right)  $, namely the arc field $X:M\times\left[
0,1\right]  \rightarrow M$ is given by%
\[
X_{t}\left(  f\right)  :=\left[  f+tG\circ\left(  K\ast f\right)  \right]
_{0}^{1}\text{.}%
\]
Assuming $G:\mathbb{R}\rightarrow\mathbb{R}$ is a Lipschitz-continuous bounded
function and $K\in L^{1}(\mathbb{R}^{n})$ Theorem \ref{ThmExUn} guarantees a
unique forward flow $F:M\times\left[  0,\infty\right)  \rightarrow M$ that is
forward-tangent to this arc field. In this section we show $F$ is the
sought-after solution of continuous Lenia by proving it solves the original
integro-differential equation. Then we investigate some elementary properties
of that solution.

\bigskip

Focusing on an initial condition $f_{0}\in M$ we define%
\[
f_{t}:=F_{t}\left(  f_{0}\right)  .
\]
For $x\in\mathbb{R}^{n}$ and $t\geq0$,%
\[
\frac{d}{dt^{+}}f_{t}\left(  x\right)  :=\lim_{h\rightarrow0^{+}}%
\frac{f_{t+h}\left(  x\right)  -f_{t}\left(  x\right)  }{h}%
\]
denotes the point-wise forward derivative of $f_{t}\left(  x\right)  $ with
respect to $t$, at the spot $x$. Note that this not a limit of functions in
the sup norm, but rather a pointwise limit of real numbers. Then we obtain the
following forward integro-differential equation:

\begin{proposition}
\label{ThmFwdD}%
\[
\frac{d}{dt^{+}}f_{t}\left(  x\right)  =\left\{
\begin{array}
[c]{cc}%
G(\left(  K\ast f_{t}\right)  (x)) & \text{if }0<f_{t}\left(  x\right)  <1\\
\left[  G\left(  (K\ast f_{t})(x)\right)  \right]  ^{0} & \text{if }%
f_{t}\left(  x\right)  =1\\
\left[  G\left(  (K\ast f_{t})(x)\right)  \right]  _{0} & \text{if }%
f_{t}\left(  x\right)  =0
\end{array}
\right.
\]
for $x\in\mathbb{R}^{n}$ and $t\geq0$.
\end{proposition}

\begin{proof}
The tangency condition $\left(  \ref{LineTangencyCondition}\right)  $ gives
pointwise
\begin{align*}
0 &  = &  &  \lim_{h\rightarrow0^{+}}%
\genfrac{|}{|}{}{0}{f_{t+h}(x)-\left[  f_{t}(x)+hG\left(  (K\ast
f_{t})(x)\right)  \right]  _{0}^{1}}{h}%
\\
&  = &  &  \lim_{h\rightarrow0^{+}}\left|  \dfrac{f_{t+h}\left(  x\right)
-f_{t}\left(  x\right)  }{h}-\dfrac{h\left[  G(\left(  K\ast f_{t}\right)
\left(  x\right)  )\right]  _{-f_{t}\left(  x\right)  /h}^{(1-f_{t}\left(
x\right)  )/h}}{h}\right|
\end{align*}
and%
\[
\lim_{h\rightarrow0^{+}}\left[  G(\left(  K\ast f_{t}\right)  \left(
x\right)  )\right]  _{-f_{t}\left(  x\right)  /h}^{(1-f_{t}\left(  x\right)
)/h}=\left\{
\begin{array}
[c]{cc}%
G(\left(  K\ast f_{t}\right)  (x)) & \text{if }0<f_{t}\left(  x\right)  <1\\
\left[  G\left(  (K\ast f_{t})(x)\right)  \right]  ^{0} & \text{if }%
f_{t}\left(  x\right)  =1\\
\left[  G\left(  (K\ast f_{t})(x)\right)  \right]  _{0} & \text{if }%
f_{t}\left(  x\right)  =0.
\end{array}
\right.
\]
holds since the argument of the clip does not depend on $h$ while the bounds
of the clip approach $0$, $+\infty$, or $-\infty$.
\end{proof}

Note that we cannot expect an equation involving the two-sided derivative of
$f_{t}(x)$, since it is easy to create examples where $f_{t}(x)$ rises from 0
to 1 with increasing speed and then abruptly comes to a stop at $f_{t}(x)=1 $.

\subsection{Basic Properties}

\begin{example}
We will show that for certain choices of $G$ and $K$, an initial condition
$f_{0}$ with bounded support can have a solution whose support explodes
instantly, i.e., $f_{t}(x)>0$ for all $x\in\mathbb{R}^{n}$ and all $t>0,$
i.e., supp$\left(  f_{t}\right)  =\mathbb{R}^{n}$.

For this, pick a kernel $K \in L^{1}(\mathbb{R}^{n})$ that is continuous,
non-negative, and satisfies $K(0) >0$; take a growth function $G$ with $G(u)
>0$ for all $u >0$ and an initial condition $f_{0}$ for which $\{y\vert
f_{0}(y) >0\}$ has non-empty interior.

Since $G((K\ast f_{t})(x))$ is non-negative and $f_{t}(x)$ is continuous in
$t$, we know by our proposition that $f_{t}(x)$ is monotone increasing as a
function of $t$. Pick $t>0$ and define $S_{t}=int\{y|f_{t}(y)>0\}$, the
interior of the set of values where $f_{t}$ is positive. Let $x$ be an element
of the boundary of $S_{t}$. By our choice of $K$ there exists $\varepsilon>0$
such that $K(z)>0$ whenever $|z|<\varepsilon$. Then there exists $y\in S_{t}$
with $|x-y|<\varepsilon$, and there is a neighborhood $U $ of $y$ where
$f_{t}>0$. By our assumptions on $K$, we have $(K\ast f_{t})(x)>0$ (integral
of a positive function over an open set is positive). This gives $G((K\ast
f_{t})(x))>0$. The expression $G((K\ast f_{t})(x))$ is continuous in $t$ and
$x$, because $G$ is continuous, $f_{t}$ depends continuously on $t$, and
convolution with $K$ is a continuous operator $L^{\infty}(\mathbb{R}%
^{n})\rightarrow L^{\infty}(\mathbb{R}^{n})$ whose image consists entirely of
continuous functions. We can thus find a neighborhood $U$ of $x$ and an open
interval $I$ containing $t$ such that $G((K\ast f_{s})(y))>0$ for all $y\in U$
and $s\in I$. Again by our proposition, and by monotonicity, this implies
$f_{t}(y)>0$ for all $y\in U$. This gives $x\in S_{t}$. So $S_{t}$ is both
open and closed: $S_{t}=\mathbb{R}^{n}$ or $S_{t}=\varnothing$. The latter is
excluded by our choice of $f_{0}$ and by monotonicity.
\end{example}

Admittedly, the previous example is contrived, since normally we would
consider growth functions $G$ with $G(0)<0$ as Chan assumes in \cite{Chan1}
and \cite{Chan2}. It does however exhibit a new effect that does not occur in
discrete Lenia, where a growth function with $G(0)=0$ clearly cannot cause
such an instantaneous support explosion.

The following proposition describes a much more typical behavior: an upper
bound to the speed with which information can travel away from the initial
condition; i.e., the support of $f_{t}$ grows at a bounded rate.

We use $\left\|  \cdot\right\|  _{2}$ to denote the Euclidean distance, or
$L^{2}$ norm, and $\left\lfloor \cdot\right\rfloor $ denotes the integer floor function.

\begin{proposition}
Assume there exists $a>0$ such that $G(u)\leq0$ for $|u|\leq a$, and let $g$
be a finite and positive upper bound for $G$. Further assume that the support
of the kernel $K$ is contained in the open Euclidean ball $B(0,R)\subseteq
\mathbb{R}^{n}$ with radius $R>0$. For $x\in\mathbb{R}^{n}$ define
\[
d(x):=\inf\{\left\|  x-y\right\|  _{2}\mid y\in\text{supp}(f_{0})\}
\]
the distance from $x$ to the support of $f_{0}$. We then have%
\[
f_{t}(x)=0
\]
for every $t$ with%
\[
0<t\leq\frac{a\left\lfloor d(x)/R\right\rfloor }{g\left\|  K\right\|  _{1}%
}\text{.}%
\]
\end{proposition}

\begin{proof}
The claim is equivalent to the statement: if $k\in\mathbb{Z}_{\geq1}$ and
$d(x)\geq kR$ and $t\leq ak/(g\left\|  K\right\|  _{1})$, then $f_{t}(x)=0$.
This can be shown by induction on $k$, and the case $k=1$ already contains the
idea. If $d(x)\geq R$, then $f_{0}$ is identically zero on $B(x,R)$, which
means $f_{t}\leq gt$ on that ball (since $f_{t}$ is continuous in $t$ and
Proposition \ref{ThmFwdD} bounds the forward derivative by $g$). This in turn
implies $|(K\ast f_{t})(x)|\leq gt\left\|  K\right\|  _{1}$. So if $gt\left\|
K\right\|  _{1}\leq a$, we'll have $G((K\ast f_{t})(x))\leq0$, meaning that
$f_{t}(x)$ cannot grow and must stay at 0.
\end{proof}

\begin{proposition}
Assume $G$ is Lipschitz and bounded and that $K\in L^{1}$.

If $f_{0}$ is continuous, then $f_{t}$ is continuous for all $t>0$.
\end{proposition}

\begin{proof}
$C(\mathbb{R}^{n},[0,1])$ is a closed subspace of $\mathcal{B}\left(
\mathbb{R}^{n},\left[  0,1\right]  \right)  $ with supremum metric $d_{\infty
}$ and our arc field restricts to this space. By completeness, the solutions
$f_{t}$ we construct are therefore all continuous functions on $\mathbb{R}%
^{n}$ provided the initial condition $f_{0}$ is.
\end{proof}

\section{Comparison with Asymptotic Lenia}

Asymptotic Lenia \cite{Kawaguchi} is an alternative model that avoids the use
of the clip function, but has also been shown to bear complicated dynamics in
simulations. Instead of using formula $\left(  \ref{LineChanFormulation}%
\right)  $, Asymptotic Lenia is defined by%
\[
f_{t+dt}=f_{t}+dt\left(  T\left(  K\ast f_{t}\right)  -f_{t}\right)
\]
which is interpreted as%
\[
\frac{d}{dt^{+}}f=V\left(  f\right)
\]
where%
\[
V\left(  f\right)  :=T\left(  K\ast f\right)  -f
\]
and $T$ is now restricted to $T:\mathbb{R}\rightarrow\left[  0,1\right]  $ and
is described as unimodal, meaning it has a single maximum in $\left(
0,1\right)  $. Now $T=\frac{G+1}{2}$ and $T\left(  0\right)  =0=T\left(
1\right)  $ is a typical assumption which guarantees solutions starting with
$0\leq f_{0}\left(  x\right)  \leq1$ to remain bounded. To see this, notice if
$f\left(  x\right)  =0$ then $T\geq0$ guarantees $V\left(  f\right)  \left(
x\right)  \geq0$ so the dynamic cannot move a positive initial condition
function negative. Similarly if $f\left(  x\right)  =1$ then since $T\leq1$ we
have $V\left(  f\right)  \left(  x\right)  \leq0$ so the dynamic cannot make
an initial function bounded by $1$ grow past $1$.

Assuming $G$ is Lipschitz with constant $C_{G}$ we then have%
\begin{align*}
\left\|  V\left(  f\right)  -V\left(  g\right)  \right\|  _{\infty}  &
=\left\|  G(K\ast f)-G(K\ast g)-\left(  f-g\right)  \right\|  _{\infty}\\
& \leq C_{G}\left\|  K\ast\left(  f-g\right)  \right\|  _{\infty}+\left\|
f-g\right\|  _{\infty}\\
& \leq\left(  C_{G}\left\|  K\right\|  _{1}+1\right)  \left\|  f-g\right\|
_{\infty}%
\end{align*}
so $V$ is Lipschitz on the Banach space $\mathcal{B}\left(  \mathbb{R}%
^{n},\mathbb{R}\right)  $. Now the traditional Picard-Lindel\"{o}f Theorem
guarantees solutions to Asymptotic Lenia which restrict to the closed subset
$\mathcal{B}\left(  \mathbb{R}^{n},\left[  0,1\right]  \right)  $. This is a
\textit{much} simpler existence proof. Moreover these solutions exist forward
and backward in time and are unique in both directions. Simulations illustrate
the intuition that such solutions would be smoother in some sense, and more
amenable to traditional mathematical analysis.

On the other hand, the original Lenia model yields unique solutions only
forward in time. This means there are more complicated and interesting
dynamics possible in original Lenia. One example is that creatures in
Asymptotic Lenia cannot ``die''. If an initial condition $f_{0}$ is not the
constant zero function, then $f_{t}$ can never become the constant $0$
function in any finite time $t<\infty$ in Asymptotic Lenia. The reason for
this is that uniqueness of solutions prevents any solution from crossing the
constant $0$ integral curve, which is a solution assuming $T\left(  0\right)
=0$. But in original Lenia the creatures \textit{can} die in finite time.
Solutions do intersect continually in the forward direction. That is very rare
using classical continuous mathematical models.

What we find even more interesting about the original Lenia model is that it
models the ``arrow of time'' in a novel way. Most physics equations admit time
symmetry, meaning any system moving forward in time admits the possibility of
moving backward in time as another valid solution. But the 2nd Law of
Thermodynamics (as well as all human experience) says that is not a realistic
model of the world. The clip function makes original Lenia different than
typical continuous dynamics because it throws away information, it displays
entropy, unlike Asymptotic Lenia. This quality of losing information is
stronger than classic models of dissipation, especially the heat equation
$f_{t}=f_{xx}$ which is partially reversible, though Lenia never is. So Lenia
reflects that aspect of reality better, yet still yields unique, continuous,
deterministic solutions \textit{forward} in time. Ilya Prigogine's concluded
that the quality of irreversibility that so many physical systems exhibit
leads us to doubt the use of deterministic models in physics, since they
typically do not display such behavior: ``The more we know about our universe,
the more difficult it becomes to believe in determinism.'' \cite{Prigogine} On
the contrary, we can model very strong irreversibility in a continuous and
deterministic dynamical system as the Lenia model proves.

\section{Extensions}

Studying dynamics on metric spaces gives a natural environment for extending
or generalizing models. For instance, we can feed Lenia's creatures, making
regions in space that help the animals to grow or starve them. For instance we
can add a food function $\phi:\mathbb{R}^{n}\rightarrow\mathbb{R}$ which
encodes the idea that there is food at location $x\in\mathbb{R}^{n}$ if
$\phi\left(  x\right)  >0$ and none if $\phi\left(  x\right)  \leq0$. Then
$X_{t}^{\phi}\left(  f\right)  \left(  x\right)  :=\left[  f\left(  x\right)
+t\left[  \phi\left(  x\right)  \right]  ^{f\left(  x\right)  }\right]
_{0}^{1}$ describes the basic growth of a creature $f$ in response to the
food. The $\left[  \phi\left(  x\right)  \right]  ^{f\left(  x\right)  }$ term
guarantees the creature $f$ doesn't grow any faster than its current strength,
so for instance it doesn't grow where $f\left(  x\right)  =0$. But because
there is no lower bound, it can starve with any specified speed at regions
where $\phi\left(  x\right)  $ is negative. Notice $X_{t}^{\phi}\left(
f\right)  :=\left[  f+t\left[  \phi\right]  ^{f}\right]  _{0}^{1}$ has
$V^{1}\left(  f\right)  :=\left[  \phi\right]  ^{f}$ which is Lipschitz, since%
\[
\left\|  V^{1}\left(  f\right)  -V^{1}\left(  g\right)  \right\|  _{\infty
}:=\left\|  \left[  \phi\right]  ^{f}-\left[  \phi\right]  ^{g}\right\|
_{\infty}\leq\left\|  f-g\right\|  _{\infty}.
\]
Therefore it has a unique local flow for all time by Theorem \ref{ThmGen1}.

Next we can combine the two dynamics%
\[
\left(  X^{2}\right)  _{t}\left(  f\right)  :=\left[  f+t\left(  \left[
\phi\right]  ^{f}+G\left(  K\ast f\right)  \right)  \right]  _{0}^{1}\text{.}%
\]
and we still have that $V^{2}\left(  f\right)  :=\left(  \left[  \phi\right]
^{f}+G\left(  K\ast f\right)  \right)  $ is Lipschitz and so the flow is
guaranteed to exist uniquely for all forward time.

Now it's easy to extend the model to a richer environment where we can keep
track of the food, so that it is depleted when it is eaten. For simplicity we
limit the food in range $\left[  a,b\right]  $. Now $M:=\mathcal{B}\left(
\mathbb{R}^{n}\times\mathbb{R}^{n},\left[  0,1\right]  \times\left[
a,b\right]  \right)  \simeq\mathcal{B}\left(  \mathbb{R}^{n},\left[
0,1\right]  \right)  \times\mathcal{B}\left(  \mathbb{R}^{n},\left[
a,b\right]  \right)  $ and, for instance,%
\[
X_{t}^{3}\left(
\begin{array}
[c]{c}%
f\\
\phi
\end{array}
\right)  :=\left(
\begin{array}
[c]{c}%
\left[  f+t\left(  \left[  \phi\right]  ^{f}+G\left(  K\ast f\right)  \right)
\right]  _{0}^{1}\text{.}\\
\left[  \phi+t\left(  -\left[  \phi\right]  ^{f}\right)  \right]  _{a}^{b}%
\end{array}
\right)  \text{.}%
\]
Now $\phi_{t}$ shrinks if $f$ eats it.

Continuing to complicate the model, we generalize the Lotka-Volterra ODE by
introducing predators and prey. Again we use $M:=\mathcal{B}\left(
\mathbb{R}^{n},\left[  0,1\right]  \right)  \times\mathcal{B}\left(
\mathbb{R}^{n},\left[  0,1\right]  \right)  $ but now allow the food itself to
move and grow like the predators. Define the predator colony $f$ as eating the
prey colony $g$ with%
\[
X_{t}^{4}\left(
\begin{array}
[c]{c}%
f\\
g
\end{array}
\right)  :=\left(
\begin{array}
[c]{c}%
\left[  f+t\left(  \left[  g\right]  ^{f}+G^{1}\left(  K^{1}\ast f\right)
\right)  \right]  _{0}^{1}\\
\left[  g+t\left(  -\left[  g\right]  ^{f}+G^{2}\left(  K^{2}\ast g\right)
\right)  \right]  _{0}^{1}%
\end{array}
\right)
\]
where we allow the predators and prey to have different growth functions
$G^{i}$ and $K^{i}$ assuming they are fundamentally different creatures with
different rules for how they evolve in time.

We can further complicate the model by allowing the predators $f$ to eat the
prey $g$ which eat the food $\phi$ with the arc field on 
$M:=\mathcal{B}
\left(  \mathbb{R}^{n},\left[  0,1\right]  \right)  \times\mathcal{B}\left(
\mathbb{R}^{n},\left[  0,1\right]  \right)  \times\mathcal{B}\left(
\mathbb{R}^{n},\left[  a,b\right]  \right)  $ 
defined by
\begin{align*}
X_{t}^{5}\left(
\begin{array}
[c]{c}%
f\\
g\\
\phi
\end{array}
\right)   & :=\left(
\begin{array}
[c]{c}%
\left[  f+t\left(  \left[  g\right]  ^{f}+G^{1}\left(  K^{1}\ast f\right)
\right)  \right]  _{0}^{1}\\
\left[  g+t\left(  -\left[  g\right]  ^{f}+\left[  \phi\right]  ^{g}%
+G^{2}\left(  K^{2}\ast g\right)  \right)  \right]  _{0}^{1}\\
\left[  \phi+t\left(  -\left[  \phi\right]  ^{g}\right)  \right]  _{a}^{b}%
\end{array}
\right) \\
& =C\left[
\begin{array}
[c]{c}%
f\\
g\\
\phi
\end{array}
+tV^{5}\left(
\begin{array}
[c]{c}%
f\\
g\\
\phi
\end{array}
\right)  \right]
\end{align*}
where the clip function $C$ is%
\[
C\left(
\begin{array}
[c]{c}%
f\\
g\\
\phi
\end{array}
\right)  =\left(
\begin{array}
[c]{c}%
\left[  f\right]  _{0}^{1}\\
\left[  g\right]  _{0}^{1}\\
\left[  \phi\right]  _{a}^{b}%
\end{array}
\right)
\]
and%
\[
V^{5}\left(
\begin{array}
[c]{c}%
f\\
g\\
\phi
\end{array}
\right)  :=\left(
\begin{array}
[c]{c}%
\left[  g\right]  ^{f}+G^{1}\left(  K^{1}\ast f\right) \\
-\left[  g\right]  ^{f}+\left[  \phi\right]  ^{g}+G^{2}\left(  K^{2}\ast
g\right) \\
-\left[  \phi\right]  ^{g}%
\end{array}
\right)
\]
which we demonstrate is Lipschitz using $\left|  \left[  x\right]  _{a}%
^{b}-\left[  y\right]  _{c}^{d}\right|  \leq\max\left\{  \left|  x-y\right|
,\left|  a-c\right|  ,\left|  b-d\right|  \right\}  $ since%
\begin{align*}
& \left\|  V\left(
\begin{array}
[c]{c}%
f^{1}\\
g^{1}\\
\phi^{1}%
\end{array}
\right)  -V\left(
\begin{array}
[c]{c}%
f^{2}\\
g^{2}\\
\phi^{2}%
\end{array}
\right)  \right\|  _{\infty}\\
& =\left\|  \left(
\begin{array}
[c]{c}%
\left[  g^{1}\right]  ^{f^{1}}+G^{1}\left(  K^{1}\ast f^{1}\right) \\
-\left[  g^{1}\right]  ^{f^{1}}+\left[  \phi^{1}\right]  ^{g^{1}}+G^{2}\left(
K^{2}\ast g^{1}\right) \\
-\left[  \phi^{1}\right]  ^{g^{1}}%
\end{array}
\right)  -\left(
\begin{array}
[c]{c}%
\left[  g^{2}\right]  ^{f^{2}}+G^{1}\left(  K^{1}\ast f^{2}\right) \\
-\left[  g^{2}\right]  ^{f^{2}}+\left[  \phi^{2}\right]  ^{g^{2}}+G^{2}\left(
K^{2}\ast g^{2}\right) \\
-\left[  \phi^{2}\right]  ^{g^{2}}%
\end{array}
\right)  \right\|  _{\infty}\\
& =\left\|  \left(
\begin{array}
[c]{c}%
\left[  g^{1}\right]  ^{f^{1}}-\left[  g^{2}\right]  ^{f^{2}}+G^{1}\left(
K^{1}\ast f^{1}\right)  -G^{1}\left(  K^{1}\ast f^{2}\right) \\
\left[  g^{2}\right]  ^{f^{2}}-\left[  g^{1}\right]  ^{f^{1}}+\left[  \phi
^{1}\right]  ^{g^{1}}-\left[  \phi^{2}\right]  ^{g^{2}}+G^{2}\left(  K^{2}\ast
g^{1}\right)  -G^{2}\left(  K^{2}\ast g^{2}\right) \\
\left[  \phi^{2}\right]  ^{g^{2}}-\left[  \phi^{1}\right]  ^{g^{1}}%
\end{array}
\right)  \right\|  _{\infty}\\
& =\max\left\{
\begin{array}
[c]{c}%
\left\|  \left[  g^{1}\right]  ^{f^{1}}-\left[  g^{2}\right]  ^{f^{2}}%
+G^{1}\left(  K^{1}\ast f^{1}\right)  -G^{1}\left(  K^{1}\ast f^{2}\right)
\right\|  _{\infty},\\
\left\|  \left[  g^{2}\right]  ^{f^{2}}-\left[  g^{1}\right]  ^{f^{1}}+\left[
\phi^{1}\right]  ^{g^{1}}-\left[  \phi^{2}\right]  ^{g^{2}}+G^{2}\left(
K^{1}\ast g^{1}\right)  -G^{2}\left(  K^{2}\ast g^{2}\right)  \right\|
_{\infty},\\
\left\|  \left[  \phi^{2}\right]  ^{g^{2}}-\left[  \phi^{1}\right]  ^{g^{1}%
}\right\|  _{\infty}%
\end{array}
\right\} \\
& \leq\max\left\{
\begin{array}
[c]{c}%
\max\left\{  \left\|  g^{1}-g^{2}\right\|  _{\infty},\left\|  f^{1}%
-f^{2}\right\|  _{\infty}\right\}  +\left\|  G^{1}\left(  K^{f^{1}}\ast
f^{1}\right)  -G^{1}\left(  K^{1}\ast f^{2}\right)  \right\|  _{\infty},\\
2\max\left\{  \left\|  g^{1}-g^{2}\right\|  _{\infty},\left\|  f^{1}%
-f^{2}\right\|  _{\infty},\left\|  \phi^{2}-\phi^{1}\right\|  _{\infty
}\right\}  +\left\|  G^{2}\left(  K^{2}\ast g^{1}\right)  -G^{2}\left(
K^{2}\ast g^{2}\right)  \right\|  _{\infty}%
\end{array}
\right\} \\
& \leq2\left\|  \left(
\begin{array}
[c]{c}%
f^{1}\\
g^{1}\\
\phi^{1}%
\end{array}
\right)  -\left(
\begin{array}
[c]{c}%
f^{2}\\
g^{2}\\
\phi^{2}%
\end{array}
\right)  \right\|  _{\infty}+\max\left\{
\begin{array}
[c]{c}%
\left\|  G^{1}\left(  K^{f^{1}}\ast f^{1}\right)  -G^{1}\left(  K^{1}\ast
f^{2}\right)  \right\|  _{\infty},\\
\left\|  G^{2}\left(  K^{2}\ast g^{1}\right)  -G^{2}\left(  K^{2}\ast
g^{2}\right)  \right\|  _{\infty},
\end{array}
\right\} \\
& \leq2\left\|  \left(
\begin{array}
[c]{c}%
f^{1}\\
g^{1}\\
\phi^{1}%
\end{array}
\right)  -\left(
\begin{array}
[c]{c}%
f^{2}\\
g^{2}\\
\phi^{2}%
\end{array}
\right)  \right\|  _{\infty}+\underset{i,j}{\max}\left\{  \left\|
\frac{dG^{i}}{dx}\right\|  _{\infty}\left\|  K^{j}\right\|  _{1}\right\}
\max\left\{  \left\|  f^{1}-f^{2}\right\|  _{\infty},\left\|  g^{1}%
-g^{2}\right\|  _{\infty}\right\} \\
& \leq\left\|  \left(
\begin{array}
[c]{c}%
f^{1}\\
g^{1}\\
\phi^{1}%
\end{array}
\right)  -\left(
\begin{array}
[c]{c}%
f^{2}\\
g^{2}\\
\phi^{2}%
\end{array}
\right)  \right\|  _{\infty}\left(  2+\underset{i,j}{\max}\left\{  \left\|
\frac{dG^{i}}{dx}\right\|  _{\infty}\left\|  K^{j}\right\|  _{1}\right\}
\right)
\end{align*}
so $V^{5}$ is Lipschitz with constant $C^{5}:=2+\underset{i,j}{\max}\left(
\left\|  \frac{dG^{i}}{dx}\right\|  _{\infty}\left\|  K^{j}\right\|
_{1}\right)  $. Therefore Theorem \ref{ThmGen2} guarantees $X^{5}$ (and
therefore all the previous simpler models $X^{1}$ to $X^{4}$) has a unique
flow for all time. Extensions with results similar to these have been explored
with simulations in \cite{Hamon}. That article and the references contained
within it provide an excellent collection of speculations on the value and
future directions of these models that are beyond the scope of the present paper.

\bigskip

Further opportunities for extending the models include time-dependence and
delay dynamics. So, for instance, combined with the previous dynamics the
growth of the creatures may not happen immediately when they encounter a
fertile land. However, by adding a delay, they might store nutrients as long
as they are in the fertile place and grow later. Or from another point of
view, they may not notice the food immediately when they encounter it. The
technical requirements are slightly more complex, but the resulting model is
still just an arc field on a metric space, so the theory naturally supports
the extension.

If we add time dependence, then Lenia becomes a faithful generalization of
neural network models from discrete to continuous. Neural cellular automata or
convolutional neural networks proceed from 2 steps: convolution, then
activation. The activation step is precisely what the $G$ function does in
Lenia. In fact, clipping is commonly used.

Fully connected neural network layers are also generalizable to continuous
models with this formalism using kernels that depend on time $s$ and space $x
$%
\[
X_{s,t}\left(  f\right)  \left(  x\right)  =\left[  f\left(  x\right)
+tG(K_{s,x}\ast f)\left(  x\right)  \right]  _{0}^{1}%
\]
where $K_{s,x}\left(  y\right)  $ depends on $x$ and $s$. Put more simply,
\[
X_{t}\left(  f\right)  =\left[  f+tC(f)\right]  _{0}^{1}%
\]
where $C:\mathcal{B}\left(  \mathbb{R}^{n},\mathbb{R}\right)  \times
\mathbb{R}\rightarrow\mathcal{B}\left(  \mathbb{R}^{n},\mathbb{R}\right)  $
because the input image $f$ is transformed to $C\left(  f\right)  $ in any way
desired in the fully connected layer. (In a convolutional layer $K_{s,x}$ is
constant in $x$.) Then given any finite set of training data, i.e., a
collection of pairs of functions $\left\{  \left(  f_{i},g_{i}\right)  |i\in
I\right\}  \subset\left(  \mathcal{B}\left(  \mathbb{R}^{n},\left[
0,1\right]  \right)  \right)  ^{2}$ we ask whether we can find a time
dependent kernel $K_{s}:\mathbb{R}^{n}\times\mathbb{R}\rightarrow\mathbb{R}$
such that the flow $F_{s,t}$ generated by the time-dependent arc field
$X_{s,t}\left(  f\right)  :=\left[  f+tK_{s}\ast f\right]  _{0}^{1}$ satisfies
$F_{0,1}\left(  f_{i}\right)  =g_{i}$ for all $i\in I$. I.e., the flow $F$
interpolates the training data. This $F$ is then the convolutional neural
network which we seek to predict the result of new inputs $f$ which will give
outputs $F_{0,1}\left(  f\right)  =g$. In this perspective we see a continuous
neural network as an infinite dimensional control problem. Given the ability
to control $K_{s}$ at any time $s$ can we find the $K_{s}$ that guarantees $F$
moves from $\left(  f_{i}\right)  $ to $\left(  g_{i}\right)  $?

\end{document}